\documentclass[12pt,letter]{amsart}

\usepackage{amssymb}
\usepackage{latexsym}
\usepackage{amsmath}
\usepackage{amscd}
\usepackage[all]{xy}

\theoremstyle{plain}
\newtheorem{theorem}{Theorem}[section]

\newtheorem{lemma}[theorem]{Lemma}
\newtheorem{conj}[theorem]{Conjecture}
\newtheorem{conjpl}[theorem]{Conjectures}
\theoremstyle{definition}

\newtheorem{defi}[theorem]{Definition}

\newtheorem{rem}[theorem]{Remark}

\newcommand{\eqdef}{\;{:=}\;}

\newcommand{\ra}{\longrightarrow}

\newcommand{\A}{{\mathbb A}}

\newcommand{\C}{{\mathbb C}}
\newcommand{\Q}{{\mathbb Q}}
\newcommand{\R}{{\mathbb R}}

\newcommand{\Z}{{\mathbb Z}}

\newcommand{\op}{\operatorname}

\newcommand{\GL}{\op{GL}}

\newcommand{\Hom}{\op{Hom}}

\newcommand{\tensor}{\otimes}

\newcommand{\G}{{\mathbb G}}

\newcommand{\res}{\op{Res}}
\newcommand{\hod}{\op{Hod}}

\newcommand{\chow}[2]{{\rm CH}^{#1}(#2)}
\newcommand{\qchow}[2]{{\rm CH}^{#1}(#2)_{{\mathbb Q}}}
\newcommand{\corr}[4]{{\rm Corr}^{#1}(#3 \times_{#2} #4)}

\begin{document}
\parindent=0pt

\title{Chow motives of universal Families over some Shimura surfaces}
\author{ Andrea Miller}
\address{Harvard University, Department of Mathematics, One Oxford Street, Cambridge, MA 02138 }
\email{millerae@math.harvard.edu}


\begin{abstract} We prove an absolute Chow--K\"unneth decomposition for the motive of universal families $\mathcal{A}$ of
 abelian varieties over some compact Shimura surfaces. We furthermore prove the Hodge conjecture for general fibres $\mathcal{A}_t$
 of $\mathcal{A}$, extending results of Ribet.
\end{abstract}

\maketitle

\section{Introduction}
In this paper we prove the existence of absolute
Chow-K\"unneth decompositions for families of abelian varieties over some compact Picard modular
surfaces thus proving a conjecture of Murre for this case. What we mean by absolute will be explained below. For more on Murre's conjecture see below and section \ref{sec:conj}.
For the ease of the reader let us recall parts of the introduction of our joint paper \cite{we} to
introduce the circle of ideas which are behind Chow--K\"unneth
decompositions. For a general reference see
\cite{murre-motive}. \\
Let $Y$ be a smooth, projective $k$ --variety of dimension $d$ and
$H^*$ a Weil cohomology theory. In this paper we will mainly be
concerned with the case $k=\C$, where we choose singular cohomology
with rational coefficients as Weil cohomology. Grothendieck's
Standard Conjecture $C$ asserts that the K\"unneth components of the
diagonal $\Delta \subset Y \times Y$  in the cohomology $H^{2d}(Y
\times Y,\Q)$  are algebraic, i.e., cohomology classes of algebraic
cycles. In the case $k=\C$ this follows from the Hodge conjecture.
Since $\Delta$ is an element in the ring of correspondences, it is
natural to ask whether these algebraic classes come from algebraic
cycles $\pi_j$ which form a complete set of orthogonal idempotents
$$
\Delta = \pi_0 +\pi_1 + \ldots +\pi_{2d} \in CH^d(Y \times Y)_\Q
$$
summing up to $\Delta$. Such a decomposition is called a {\sl
Chow--K\"unneth decomposition} and it is conjectured to exist for
every smooth, projective variety. One may view $\pi_j$ as a Chow
motive representing the projection onto the $j$--the cohomology
group in an universal way. There is also a corresponding notion for
$k$--varieties which are relatively smooth over a base scheme $S$.
See also section \ref{sec:conj}, where Murre's refinement of this conjecture with
regard to the Bloch--Beilinson filtration is discussed.
Chow--K\"unneth decompositions for abelian varieties were first
constructed by Shermenev and later more generally over any base scheme by Deninger and Murre (\cite{den-mur}). Building on work of Beauville, they use Fourier-Mukai transforms to construct projectors, see \cite{den-mur} and section \ref{sec:conj}. Fourier--Mukai transforms may be
effectively used to write down the projectors, see also
\cite{motives-kuenne}. The case of surfaces was
treated by Murre~\cite{murre-surface}. He in particular gave a
general method to construct the projectors $\pi_1$ and $\pi_{2d-1}$,
leading to the so--called Picard and Albanese Motives. Some special
classes of $3$--folds were considered in \cite{dAMS}.\\
The line of work on modular varieties was begun by Gordon and Murre in \cite{gordon-murre} on elliptic modular threefolds and by Gordon, Hanamura and Murre in \cite{gordon-hanamura-murre-i, gordon-hanamura-murre-iii} on Hilbert modular varieties.
 As these two cases represent the two most well-known and extensively studied Shimura varieties and many fundamental
 connections to modular forms and $\ell$-adic Galois representations are known and studied, it was natural to undertake an attempt to treat "the next important" Shimura variety (depending on taste associated to $GU(n,1)$ or to $GSp_4$ ).
 But as shown in \cite{we} even in a very specific case of a carefully chosen single Picard modular surface, involving many special choices,
 the assumptions of the main theorem of   \cite{gordon-hanamura-murre-i, gordon-hanamura-murre-iii} proved to be too strong and a Chow-K\"unneth decomposition could not be obtained. In a way the main result of \cite{we} was to show that the methods of \cite{gordon-hanamura-murre-i, gordon-hanamura-murre-iii} fail and how the difficulties increase as soon as one moves away from $Gl_2$.
 One of the side results of the present paper shows that theorem 1.3. of \cite{gordon-hanamura-murre-iii} also cannot be applied (for different reasons than in \cite{we}) in the case of the compact Picard modular surfaces we are considering. Even though the approach of \cite{gordon-hanamura-murre-iii} seems to fail for modular varieties which are neither elliptic nor Hilbert, it were the ideas of this approach
 that enabled us to attack absolute Chow-K\"unneth decompositions for the modular varieties treated in this paper. We will now give a more detailed description of the problem and our method of attack.\\

Let us assume that we have a family $f:A \to X$ of abelian varieties over a smoth variety $X$ (which will be a Shimura variety in our case). Since all fibers are abelian,
we obtain a relative Chow---K\"unneth decomposition over $X$ in the sense of Deninger and Murre (\cite{den-mur}),
i.e., algebraic cycles $\Pi_j$ in
$A \times_X A$ which sum up to the diagonal $\Delta_{A/X}$. One may view $\Pi_j$ as a projector related to
$R^jf_*\C$. Now let $\overline{f}: \overline{A} \to \overline{X}$ be a compactification of the family.
We will, for a moment, use the language of perverse sheaves as in \cite{BBD}, in particular the notion of a stratified map.

In \cite{corti-hanamura} Corti and Hanamura have formulated a motivic analogue of the topological decomposition theorem of \cite{BBD}. In \cite{gordon-hanamura-murre-ii} Gordon, Hanamura and Murre have formulated this as {\sl Motivic Decomposition Conjecture}:

\begin{conj} \label{MDC} Let $\overline{A}$ and $\overline{X}$ be quasi--projective varieties over $\C$, $\overline{A}$
smooth, and $\overline{f}: \overline{A} \to \overline{X}$ a projective map. Let $\overline{X}=X_0 \supset X_1 \supset
\ldots \supset X_{\dim(X)}$ be a stratification of $\overline{X}$ so that $\overline{f}$ is a stratified map. Then there are local
systems ${\mathcal V}^j_\alpha$ on $X_\alpha^0=X_\alpha \setminus X_{\alpha-1}$, a complete set
$\Pi^j_\alpha$ of orthogonal projectors and isomorphisms
$$
\sum_{j,\alpha} \Psi^j_\alpha: R\overline{f}_*\Q_{\overline{A}} {\buildrel \cong \over \to}
\bigoplus_{j,\alpha} IC_{X_\alpha}( {\mathcal V}^j_\alpha)[-j-\dim(X_\alpha)]
$$
in the derived category.
\end{conj}

This conjecture asserts of course more than a relative Chow--K\"unneth decomposition for the smooth part $f$
of the morphism $\overline{f}$. Due to the generally complicated structure of the strata its proof in special cases generally needs some more
information about the geometry of the stratified morphism $\overline {f}$.
In our case we can bypass this difficulty since we are working with projective varieties. We further know that in our case a Chow-K\"unneth decomposition exists for the base modular variety $\overline{X}=X$ by Murre's result \cite{murre-surface} on surfaces. By ''absolute'' Chow-K\"unneth decomposition of $\overline{A}$ we mean a Chow-K\"unneth decomposition of $\overline{A}\ra\overline{X}\ra \op{Spec}k$, ( $\overline{X}$ is defined over $\op{Spec}k$). We are now left with the two main difficulties of the construction of Chow-K\"unneth decompositions.
\begin{enumerate}
\item Obtain information on the image of the cycle class map for the fibres of $\overline{A}\ra\overline{X}.$
\item Obtain as much vanishing as possible for the cohomology  $H^*(X,\mathcal{V})$, where $\mathcal{V}$ is a local coefficient system in order to make the Leray spectral sequence of the fibration $\overline{A}\ra\overline{X}$ as computable as possible.
\end{enumerate}
In \cite{gordon-hanamura-murre-iii} difficulty 1 was tackled by using Ribet's result \cite{ribet}. Difficulty 2 didn't exit by the vanishing theorem of Matsushima-Shimura (\cite{matsushima-shimura}).
In \cite{we} difficulty 1 could still be handled by Ribet's results, but difficulty 2 couldn't be handled, because for the non-compact Picard modular surface considered $H^*(X,\mathcal{V})$ doesn't vanish outside of the middle degree. In the present paper we will extend Ribet's results in order to deal with difficulty 1, and difficulty 2 is dealt with by checking the vanishing of $H^*(X,\mathcal{V})$ for $*\neq 2$ for a very specific local coefficient system case by case by hand.\\
A new approach of this paper compared to previous work on this subject consists in working with (special) Mumford-Tate groups instead of monodromy groups. In our case this makes the crucial difference for the vanishing of the cohomology of the underlying Shimura variety (see section 6).

It should be interesting to look at cases where neither Ribet's $(1,1)$-criterion holds nor vanishing outside of the middle degree exists.

\bigskip

The contents of this paper are organized as follows: In section \ref{sec:conj} we recall Murre's conjectures and the construction of Chow-K\"unneth projectors for abelian varieties. In section \ref{sec:surface} we describe the Shimura surfaces we are working with and recall their modular interpretation. In section \ref{sec:general} we formulate a slight generalization of theorem 1.3 of \cite{gordon-hanamura-murre-iii} and prove it. In section \ref{sec:hodge} we prove the Hodge conjecture for a certain class of abelian varieties arising from our moduli problem. Here we extend work of Ribet (see \cite{ribet}). Knowing the Hodge conjecture will enable us to provide some of the assumptions needed in theorem \ref{ghm2}.
To deal with another assumption of theorem \ref{ghm2} we will have to
provide the vanishing of certain cohomology groups of the Shimura surface. This is done in section \ref{sec:cohomology}. Finally in section \ref{sec:proof} we put everything together to prove the existence of absolute Chow-K\"unneth projectors for our Shimura surfaces.

\section{Standard Conjecture ´´C´´ and Murre's Conjecture}\label{sec:conj}

Let us briefly recall some definitions, results and conjectures from the theory
of Chow motives. We refer to \cite{murre-motive} for details.
\subsection{}
For a smooth projective variety $Y$ over a field $k$ let
$\chow{j}{Y}$ denote the Chow group of algebraic cycles of
codimension $j$ on $Y$ modulo rational equivalence, and let
$\qchow{j}{Y}\eqdef\chow{j}{Y}\otimes\Q.$ For a cycle $Z$ on $Y$
we write $[Z]$ for its class in $\chow{j}{Y}.$ We will be working
with relative Chow motives as well, so let us fix a smooth
connected, quasi-projective base scheme $S\to \op{Spec}k.$ If
$S=\op{Spec}k$, we will usually omit $S$ in the notation. Let
$Y,Y'$ be smooth projective varieties over $S$, i.e., all fibers are smooth.
For ease of notation (and as we will not consider more general cases) we may
assume that $Y$ is irreducible and of relative dimension $g$ over
$S.$ The group of relative correspondences from $Y$ to $Y'$ of
degree $r$ is defined as
\[\corr{r}{S}{Y}{Y'}\eqdef \qchow{r+g}{Y\times_S Y'}.\]
Every $S$-morphism $Y'\to Y$ defines an element in
$\corr{0}{S}{Y}{Y'}$ via the class of the transpose of its graph.
In particular one has the class
$[\Delta_{Y/S}]\in\corr{0}{S}{Y}{Y}$ of the relative diagonal.
The self correspondences of degree $0$ form a ring, see \cite[p. 127]{murre-motive}.
Using the relative correspondences one proceeds as usual to define
the category $CH\mathcal{M}(S)$ of (pure) Chow motives over $S.$ The
objects of this pseudoabelian $\Q$-linear tensor category are
triples $(Y,p,n)$ where $Y$ is as above, $p$ is a projector, i.e.
an idempotent element in $\corr{0}{S}{Y}{Y},$ and $n\in\Z.$ The
morphisms are
\[\Hom_{\mathcal{M}_S}((Y,p,n),(Y',p',n'))
\eqdef p'\circ \corr{n'-n}{S}{Y}{Y'}\circ p.\]
When $n=0$ we write $(Y,p)$ instead of $(Y,p,0),$ and $h(Y)\eqdef
(Y,[\Delta_Y]).$

\subsection{Murre's conjectures}

Now denote by $\mathfrak{SP}/k$ the category of smooth projective varieties over a field $k$.

Fix a Weil cohomology theory $H^*$ and $Y\in\mathfrak{SP}/k$ and let $\op{dim} Y=d$. Then by the K\"unneth formula we have
$$H^{2d}(Y\times Y)(d)=\oplus _{n=0}^{2d} H^n(Y)\tensor H^{2d-n}(Y)(d).$$
Since $Y$ is projective, by Poincar\'e duality, we have
$$H^{2d-n}(Y)(d)=H^n (Y)^{\vee},$$
where $\vee$ denotes dual cohomology. We then get
\begin{eqnarray}
H^{2d}(Y\times Y)(d)&=&\oplus _{n=0}^{2d} H^n(Y)\tensor H^n (Y)^{\vee}\nonumber \\ &=&\oplus _{n=0}^{2d} Hom (H^n(Y), H^n (Y)).
\nonumber\end{eqnarray}
We can thus identify $H^{2d}(Y\times Y)(d)$ with the vector space of graded $k$-linear maps $f:H^*(Y)\ra H^* (Y)$. In particular we can  write
$$id_{H^*(Y)}=\sum _{n=0}^{2d} \overline{\pi} _{Y}^n \quad\textrm{ where }\overline{\pi} _{Y}^n\in H^n(Y)\tensor H^n (Y)^{\vee}.$$

The projector $$\overline{\pi} _{Y}^n : H^*(Y)\ra H^* (Y)$$ is called the {\sl n-th K\"unneth projector}.

\begin{conj}[Grothendieck]
The K\"unneth projectors
$$\overline{\pi} _{Y}^n : H^*(Y)\ra H^* (Y), n=0\ldots 2d$$
are algebraic.
\end{conj}

Since the diagonal $[\Delta_{Y/k}]\in\chow{d}{Y\times Y}$ is mapped to $id_{H^*(Y)}$ by the cycle class map we can ask further if the $\overline{\pi} _{Y}^n ,\; (n=0,\ldots 2d)$ lift to orthogonal projectors
$\pi_0 ,\ldots ,\pi_{2d} \in \chow{d}{Y\times Y}.$

\begin{defi}
For a smooth projective variety $Y/k$ of dimension $d$ a
{\sl Chow-K{\"u}nneth-decomposition} of $Y$ consists of a collection of
pairwise orthogonal projectors $\pi_0,\ldots,\pi_{2d}$ in
$\corr{0}{}{Y}{Y}$ satisfying
\begin{enumerate}
\item $\pi_0+\ldots +\pi_{2d}=[\Delta_Y]$ and
\item for some Weil cohomology theory $H^*$ one has $\pi_i(H^*(Y))=H^i(Y).$
\end{enumerate}
\end{defi}
If one has a Chow-K{\"u}nneth decomposition for $Y$ one writes
$h^i(Y)=(Y,\pi_i)$. A similar notion of   a {\sl relative
Chow-K{\"u}nneth-decomposition} over $S$ can be defined in a straightforward manner, see \cite{murre-motive}.

The existence of such a decomposition for every smooth projective variety is part of the following
conjectures of Murre:

\begin{conjpl}
Let $Y$ be a smooth projective variety of dimension $d$ over some
field $k.$
\begin{enumerate}
\item There exists a Chow-K\"unneth decomposition for $Y$.
\item For all $i<j$ and $i>2j$ the action
of $\pi_i$ on $CH^j(Y)_\Q$ is trivial, i.e. $\pi_i\cdot CH^j(Y)_\Q=0$.
\item The induced $j$ step filtration on
$$
F^\nu CH^j(Y)_\Q:= {\rm Ker}\pi_{2j} \cap \cdots \cap
{\rm Ker}\pi_{2j-\nu+1}
$$
is independent of the choice of the
Chow--K\"unneth projectors, which are in general not canonical.
\item The first step of this filtration should give exactly the
subgroup of homological trivial cycles $CH^j(Y)_\Q$ in $CH^j(Y)_\Q$.
\end{enumerate}
\end{conjpl}

\begin{rem}
U. Jannsen showed in \cite{jannsen} that Murre's conjectures are equivalent to the Bloch-Beilinson conjecture.
\end{rem}

There are not many examples for which these conjectures have been
proved, but they are known to be true for surfaces
\cite{murre-motive}.
Furthermore for abelian schemes over  a smooth projective base Deninger and Murre have constructed relative projectors in \cite{den-mur}, generalizing work of Shermenev and Beauville.

\subsection{Projectors for Abelian Varieties}
Let $S$ be a fixed base scheme. We will now
state the results on relative Chow motives in the case that $A$
is an abelian scheme of fibre dimension $g$ over $S.$
Firstly we have a functorial decomposition of the relative
diagonal $\Delta_{A/S}.$
\begin{theorem}
\label{uniquedec}
There is a unique decomposition
\[\Delta_{A/S}= \sum_{s=0}^{2g}\Pi_i\qquad\text{in}\qquad \qchow{g}{A\times_S A}\]
such that $(id_A\times [n])^*\Pi_i=n^i\Pi_i$ for all $n\in\Z.$
Moreover the $\Pi_i$ are mutually orthogonal idempotents, and
$[{}^t\Gamma_{[n]}]\circ\Pi_i=
n^i\Pi_i=\Pi_i\circ[{}^t\Gamma_{[n]}],$ where $[n]$ denotes the
multiplication by $n$ on $A.$
\end{theorem}
\begin{proof}
\cite[Thm.~3.1]{den-mur}
\end{proof}

Putting $h^i(A/S)=(A/S,\Pi_i)$ one has a Poincar\'{e}-duality for
these motives.

\begin{theorem}{\em (Poincar\'e-duality)}
\[h^{2g-i}(A/S)^\vee\simeq h^{i}(A/S)(g)\]
\end{theorem}
\begin{proof}
\cite[3.1.2]{motives-kuenne}
\end{proof}

\section{Kottwitz's compact Shimura surfaces and their modular interpretation}\label{sec:surface}

For everything in this section see
\cite{harris-taylor} and \cite{kottwitz2}, where unitary groups of arbitrary signature are treated. We have followed Clozel's convention in \cite{clozel} to call projective Shimura varieties constructed in the following  particular way Kottwitz's Shimura varieties. It should be noted though that a vanishing theorem for these varieties coming from signature $(2,1)$ was obtained much earlier by Rapoport and Zink in \cite{rapoport-zink}.\\

Let $E/\Q$ be an imaginary quadratic extension. Denote by $\A_f$ the finite $\Q$-adeles. Let $(D,*)$ be a division algebra of dimension 9 over $E$ with center $E$ together with an involution $*$  which induces on $E$ the non-trivial automorphism of $E/\Q$, e.g. is of the second kind. Let $G$ be the $\Q$-group whose points in any commutative $\Q$-algebra $R$ are given by
$$G(R)=\{ x\in D\tensor_{\Q}R\mid xx^*\in R^*\}.$$

Let $h$ be an $\R$-algebra homomorphism
$$h:\C\ra D\tensor_{\Q}\R$$
such that $h(z)^*=h(\overline{z})$ for all $z\in\C$. Then the map
$$x\mapsto h(i)^{-1}x^* h(i)$$
is an involution of $D\tensor_{\Q}\R$. Assume that this involution is positive. By abuse of notation we will denote the restriction of $h$ to $\C ^*$ also by $h$. This gives us a homomorphism of $\R$-groups
$$h:\C^*\ra\G(\R).$$
Let $X_{\infty}$ denote the $G(\R)$-conjugacy class of $h$. We thus get
$$X_{\infty}=G(\R)/K_{\infty},$$
where $K_{\infty}$ is the centralizer of $h$ in $G(\R)$. Kottwitz showed in \cite{kottwitz2} (Lemma 4.1) that the pair $(G,X_{\infty})$ satisfies Deligne's  conditions (2.1.1.1)--(2.1.1.5) of \cite{deligne2} and therefore gives rise to a Shimura variety $S_K$ for every compact open subgroup $K\subset G(\A_f)$. It can be shown that $G$ is an inner form of the unitary similitude  group $GU(2,1)$ of signature (2,1). We also assume $K$ small enough so that $S_K$ is smooth.
Furthermore it is shown in \cite{kottwitz2} that $S_K$ is projective.

It follows from \cite{deligne3} that $S_K$ is the moduli space for
certain isogeny classes of abelian varieties with polarization, endomorphisms and level structure. More explicitly it is shown in \cite{kottwitz2}, \S 5 that $S_K$ is the moduli space for quadruples $(A,\lambda, i, \overline{\eta})$ satisfying Kottwitz's determinant condition (loc.cit.). Here $A$ denotes an abelian variety of dimension 9 with multiplication by $D$, $\lambda$ denotes a polarization, $i:D\hookrightarrow End^0 (A)$ a $*$-homomorphism
 and finally $\overline{\eta}$ a level structure. We use the standard notation $End^0 (A)=End(A)\tensor\Q$.

\section{A general theorem}\label{sec:general}

In this section we prove theorem \ref{ghm2} below.
We restrict ourself to the situation needed in the case described in section \ref{sec:surface}, namely a compact one, even though the proof goes through in the non-compact case. Yet we need to weaken the assumptions on monodromy of \cite{gordon-hanamura-murre-iii} since these assumptions are false in our case (and most other cases of PEL-Shimura varieties).
 Generally the prerequisites of theorem 1.3. of \cite{gordon-hanamura-murre-iii} are too strong even in the compact case, so there is very little hope to being able to use it for non-compact Shimura varieties beyond the ones studied in \cite{gordon-hanamura-murre-iii}. See \cite{we} for a non-trivial low-dimensional case where the prerequisites of theorem 1.3. of \cite{gordon-hanamura-murre-iii} fail.

The following two lemmas will be needed in the proof of theorem \ref{ghm2}.

\begin{lemma}\label{lemma:proj}
Let $X$ be a projective variety over an algebraically closed field $k$, then there is a functor
$$\pi_{*}: CH\mathcal{M}(X)\ra CH\mathcal{M}(k)$$
where $\pi_{*}(\mathcal{A}/X,r,P)=(\mathcal{A},r,P)$ and the induced maps on $\op{Hom}$ groups are $CH_{*}(\mathcal{A}\times_X \mathcal{A'})\ra CH_{*}(\mathcal{A}\times \mathcal{A'}).$
\end{lemma}

\begin{proof}
This is part of proposition 1.1. of \cite{gordon-hanamura-murre-i}.
\end{proof}

\begin{lemma}\label{lemma:fund}
Let $X$, $S$ be quasi-projective varieties over an algebraically closed field $k$, with $X$ smooth over $k$, and $p:X\to S$ a projective map.
\begin{enumerate}

\item[(a)] Giving a projector $\Pi\in CH_{dim X}(X\times_S X)$ and an isomorphism in $CH\mathcal{M}(S)$
$$f:(X,\Pi,0)\ra \oplus^m (S,\Delta (S),-q)$$
is equivalent to giving elements
$$f_1,\ldots,f_m\in CH_{dim S+q}(X),\quad g_1,\ldots,g_m\in CH_{dim S-q}(X)$$
subject to the condition $p_*(f_i\cdot g_j)=\delta _{ij}[S].$ Here $f_i\cdot g_j$ is the intersection product in $X$, and $[S]\in CH_{dim S}(S)$ is the fundamental class.

\item[(b)] If $(X,P,0)\in CH\mathcal{M}(S)$ and $\Pi$ is a constituent of $P$ (i.e. $P\circ\Pi =\Pi\circ P =\Pi $) then we can take the above $f_i$ and $g_j$ which moreover satisfy $f_i\circ P =f_i$ and $P\circ g_j = g_j$. Conversely, if we have such $f_i$ and $g_j$, then the corresponding $\Pi$ is a constituent of $P$.

\end{enumerate}

\end{lemma}

\begin{proof}
This is Lemma 4 of \cite{gordon-hanamura-murre-iii}. See pp. 144 of \cite{gordon-hanamura-murre-iii} for its proof.
\end{proof}

\subsection{The Hodge structure and Hodge group of an abelian variety}
Instead of working with fundamental groups and monodromy representations, as it was done in previous work on this subject, we will systematically use special Mumford-Tate groups. Note that we will use Mumford's definition of these groups (see \cite{mumford}) (and not Deligne's generalized definition as defined in \cite{deligne}). We will now recall all necessary definitions.\\
 Let $\mathcal{A}_t$ be an abelian variety over $\C$ (a generic fibre of $\mathcal{A}$ in our case). Then the Betti homology group $V=H_1^B(\mathcal{A}_t,\Q)$ carries a Hodge
structure of weight $-1$, and type $\{(-1,0),(0,-1)\}$. This means that there
exists a canonical decomposition:
$$V\otimes\mathbb C=V^{-1,0}\oplus V^{0,-1}$$
satisfying
$$\overline{V^{-1,0}}=V^{0,-1},$$
More generally Hodge structures of any weight and type are
described by an action $h$ of the algebraic group $\mathbb S=\res_{\mathbb C/\mathbb R}\mathbb G_m$ (often called Deligne torus in the context of Shimura varieties) on the underlying real vector space
$$V\otimes\mathbb R.$$ This means that $z$ acts as
$z^{-p}\overline{z}^{-q}$ on the $V^{p,q}$ piece (the (-)-normalization
arises as $V^{-1,0}$ is the tangent space of $A$).\\

Let $(V,h: \mathbb S\ra\GL (V_{\R}))$ be a rational pure Hodge structure of fixed weight.

\begin{defi}
The {\sl Mumford-Tate} group $MT(V,h)$ of a rational Hodge structure $(V,h)$ is the smallest algebraic $\mathbb Q$-subgroup of
$GL(V)$ such that $MT(V,h)\times\mathbb R$ contains the image of $h$.
\end{defi}

Let us denote by $\mathbb S^1$ the kernel of the norm from $\mathbb S$ to
$\mathbb G_m$, and let us write $h^1$ for the restriction of $h$ to $\mathbb S^1$.

\begin{defi}
The {\sl Hodge} group (often called {\sl special Mumford-Tate group}) $\hod (V,h)$ of a rational Hodge structure $(V,h)$ is the smallest algebraic $\mathbb Q$-subgroup of
$GL(V)$ such that $\hod(V,h)\times\mathbb R$ contains the image of $h^1$.
\end{defi}

\begin{rem}\label{rem}
For a rational pure Hodge structure of fixed weight we have $$MT(V,h)=\mathbb G_m\cdot \hod(V,h).$$
\end{rem}

\subsection{Proof of theorem \ref{ghm2}}

We are now ready to prove the following theorem.

\begin{theorem}
\label{ghm2} Let $X$ be the $\C$-valued points of a smooth and projective PEL-Shimura variety and let $p:\mathcal{A}\to X$ be the universal abelian scheme over $X$. Assume the following conditions to hold:

\begin{enumerate}

\item The scheme $\mathcal{A}/X$ has a relative Chow-K\"{u}nneth decomposition.
\item $X$ has a Chow-K\"{u}nneth decomposition over $\C$.
\item If $t$ is a general point of $X$ , the natural map.
\[CH^r(\mathcal{A})\to H^{2r}_B(\mathcal{A}_t(\mathbb{C}),\mathbb{Q})^{\hod (\mathcal{A}_t)}\]
is surjective for $0\le r\le d=\dim \mathcal{A}-\dim X$.
\item  Let $\rho$ be an irreducible, non-constant
representation of ${\hod (\mathcal{A}_t)}$ and $\mathcal{V}$ the
corresponding local system on $X$.

Then the cohomology $H^q(X,\mathcal{V})$ vanishes if
$q\neq\dim X$.
\end{enumerate}
Under these assumptions $\mathcal{A}$ has a Chow-K\"{u}nneth
decomposition over $\C$.

\end{theorem}

\begin{rem}
By general point in condition (3) we mean arising by base change from a generic point of a model of
$X$ over its function field.\\
The proof of the theorem will show that it is
sufficient to assume condition (4) for a very specific set of local systems $\mathcal{V}^i$.

Note that Condition (1) always holds by the above mentioned results of Deninger and Murre.
\end{rem}

\begin{proof}
By assumption (1) of the theorem there is a relative Chow-K\"unneth decomposition for $\mathcal{A}\to X$. If $d$ is the relative dimension of $\mathcal{A}$ over $X$, denote by $P^0,\ldots ,P^{2d}$ the relative Chow-K\"unneth projectors for $\mathcal{A}/X$. We will decompose the even projectors $P^{2r}$ into an algebraic and a transcendental part.\\
Let $t\in X$ be generic. There is a non-degenerate pairing
$$H_B^j(\mathcal{A}_t,\Q)\times H_B^{2d-j}(\mathcal{A}_t,\Q)\ra\Q,\; j=0\ldots 2d .$$
Deligne showed in \cite{deligne} that $\hod (\mathcal{A}_t)$ acts semi-simply on $H_B^{*}(\mathcal{A}_t,\Q)$. In particular we get a non-degenerate pairing
$$H_B^{2r}(\mathcal{A}_t,\Q)^{\hod (\mathcal{A}_t)}\times H_B^{2d-2r}(\mathcal{A}_t,\Q)^{\hod (\mathcal{A}_t)}\ra\Q,\; r=0\ldots d .$$
Choose a base for $H_B^{2r}(\mathcal{A}_t,\Q)^{\hod (\mathcal{A}_t)}$ and a dual base for $H_B^{2d-2r}(\mathcal{A}_t,\Q)^{\hod (\mathcal{A}_t)}$. By assumption (3) we can lift the elements of the base and the elements of the dual base to elements
$$g_j\in CH^r(\mathcal{A}),\quad f_i\in CH^{d-r}(\mathcal{A}),$$ for $j=1,\ldots ,m=\op{dim}\; H_B^{2r}(\mathcal{A}_t,\Q)^{\hod (\mathcal{A}_t)}.$\\
Note that $P^{2r}\circ g_j=g_j \quad(*)$ since $g_j$ maps to $H_B^{2r}(\mathcal{A}_t,\Q)^{\hod (\mathcal{A}_t)}$ and $P^{2r}$ is the $2r$-th relative Chow-K\"unneth projector. We get $f_i\circ P^{2r}=f_i$ by dualizing: From $(*)$ we obtain $(P^{2r}\circ g_j)^{\vee}=g_j^{\vee}=f_j $. Furthermore $(P^{2r}\circ g_j)^{\vee}=f_j\circ (P^{2r})^{\vee}=f_j\circ (P^{2d-2r})$ by Lieberman's lemma.\\
By lemma \ref{lemma:fund} we can use the $\{f_i \},\{g_j \}_{i,j=1,\ldots ,dim H_B^{2r}(\mathcal{A}_t,\Q)^{\hod (\mathcal{A}_t)}}$ to construct constituents $P^{2r}_{alg}$ of $P^{2r},\;r=0,\ldots ,d$.\\
Set $P^{2r}_{trans}:=P^{2r}-P^{2r}_{alg}$.\\
By lemma \ref{lemma:proj} we know that $P^{2r-1}, P^{2r}_{alg}, P^{2r}_{trans}, r=0,\ldots ,d$ extend to absolute projectors of $\mathcal{A}/k$.
We now have to decide which of these projectors are pure, meaning which project to cohomology (of $\mathcal{A}$) of one single degree.\\

To this goal we will now consider the following local systems $\mathcal{V}^i$ on $X$, defined by
$$(\mathcal{V}^i)_t(=\left\{
\begin{array}{ll}
(R^{i}p_*\Q)_t & \textrm{if $i=2r-1, i=0,\ldots ,d$,}\\
(R^{i}p_*\Q )_t/(R^{i}p_*\Q)_t^{\hod (\mathcal{A}_t)} & \textrm{if $i=2r, i=0,\ldots ,d$,}\\
\end{array}\right.
$$

where by $(R^{i}p_*\Q)_t^{\hod (\mathcal{A}_t)}$ we denote the constant system with fiber $H_B^{i}(\mathcal{A}_t,\Q)^{\hod (\mathcal{A}_t)}$. Now for $i=2r-1$ we know that $(R^{i}p_*\Q)_t)^{\hod (\mathcal{A}_t)}$ vanishes as Hodge cycles are even. Clearly for $i=2r$, by construction, $(R^{i}p_*\Q)_t /(R^{i}p_*\Q)_t^{\hod (\mathcal{A}_t)}$ doesn't contain any $\hod (\mathcal{A}_t)$-invariants either. Thus $(\mathcal{V}^i)_t$ doesn't contain any $\hod (\mathcal{A}_t)$-invariants and by assumption (4) of our theorem we obtain that $H^{q}(X,\mathcal{V}^i ) =0 \textrm{ for }q\neq \op{dim} X.$

Now denote by $M^{i}_?$ the motive cut out by the projector $P^{i}_?$, where $?\in\{ \; alg, trans, \ldots\}$.\\
Recall that by construction (assumption (1)) $P^{2r-1}R^{\bullet}p_*\Q$ is pure of degree $2r-1,\; r=0,\ldots ,d$. But as we have just shown this means that $H^{q}(X, P^{2r-1}R^{\bullet}p_*\Q)=0 \textrm{ for }q\neq \op{dim} X.$ Hence $M^{2r-1}\; r=0,\ldots ,d$ is pure and $P^{2r-1}$ is an absolute projector.\\
The same argument works for $M^{2r}_{trans}$ since by definition we have $M^{2r}=M^{2r}_{alg} + M^{2r}_{trans}$ and by construction the $P^{2r}_{alg}Rp_*\Q$ contain all Hodge cycles for $\mathcal{A}_t$. We conclude that $P^{2r}_{trans} Rp_*\Q$ does not contain any Hodge cycles (i.e. any $\hod (\mathcal{A}_t)$-invariants). As above we get that $M^{2r}_{trans}$ is pure and $P^{2r}_{trans}$ is an absolute projector.\\
As for the $P^{2r}_{alg}$ we apply lemma \ref{lemma:fund} to obtain
an isomorphism (in $CH\mathcal{M}(X)$)
$$f:(\mathcal{A},P^{2r}_{alg},0)\ra \oplus^m (X,\Delta (X),-q).$$
By assumption (2) of our theorem each summand on the right possesses a Chow-K\"unneth decomposition. Via the isomorphism we get a decomposition of $(\mathcal{A},P^{2r}_{alg},0)$ into pure motives, hence a decomposition of $P^{2r}_{alg}$ into absolute Chow-K\"unneth projectors. This completes the proof of the theorem.

\end{proof}

\section{Proof of the Hodge cojecture for $\mathcal{A}_t$}\label{sec:hodge}
In this section we will prove the Hodge conjecture for the fibres $\mathcal{A}_t$ described in the previous section. The proof builds on work of Takeev and Ribet and extends the class of abelian varieties considered by Ribet in \cite{ribet}.\\
I wish to remark, that after putting this paper on \rm{xxx.lanl.gov},  \\
S.Abdulali, whom I thank very much for pointing this out to me,  told me that the case I treat here was already treated by K.Murty in \cite{murty}.
Our proof differs though, as we make use of the structure of the underlying Shimura variety.
\bigskip

Let $\mathcal{A}$ and $X$ be as in section \ref{sec:surface}.

As before let $V=H_1^B(\mathcal{A}_t,\Q)$. Denote by $\psi$ a non-degenerate form on $V$ coming from a polarization on $\mathcal{A}_t$.
We now need to define an additional group $\frak{G}$ associated to $V$.
\begin{defi}
Let $\frak{G}$ be the largest connected subgroup of $\GL _V$ which commutes with $End^0 (A)$ and which is contained in $\op{GSp}(V,\psi)$ for a fixed $\psi$.
\end{defi}

\begin{rem}
See \cite{ribet} for the fact that $\frak{G}$ is in fact independent of a choice of $\psi$. It is also shown there that $MT(A)\subset\frak{G}$.
\end{rem}
We now return to our particular situation
Let $E$ be an imaginary quadratic field and let $B$ be a central simple
division algebra of degree $3^2$ over $E$. Let $*$ be a positive involution of
the second kind.

We will prove a slight generalization of Ribet's Theorem 0 of \cite{ribet}.

\begin{theorem}\label{0}
Let $A$ be an abelian variety whose (rational) endomorphism algebra $End^0 (A)$ is a central division algebra with an involution of the second kind and whose center is an imaginary quadratic field. Suppose that the Mumford-Tate group $MT(A)$ equals the group $\frak{G}$ defined above. Then all powers $A^m \; (m\geq 1)$ of $A$ satisfy the $(1,1)$ criterion.
\end{theorem}
\begin{proof}
Ribet proved this theorem for the case when $End^0 (A)$ is a commutative field. His proof distinguishes between the totally real case and the CM-case. Our case is a direct adaptation of the CM case to the division algebras described above. The modifications go as follows (see pp. 531 of \cite{ribet}): as totally real field take $\Q$, replace the commutative field by the center of $End^0 (A)=B$
(which is the imaginary quadratic field $E$), note that since $B$ comes equipped with an involution of the second kind, the requirements on the action on $E$ holds. Just as shown in section \ref{sec:surface} we obtain a general unitary group $\op{GU}(2,1)$ of signature $(2,1)$. Using the assumption we get $MT(A)=\op{GU}(2,1)$. Because of \ref{rem} we get $Hod(A)=\op{U}(2,1)$. We do not need to decompose $V$ into spaces $X_{\sigma}$ (notation of \cite{ribet}) since we work over $\Q$. The proof that $\bigwedge _{\R}^* (V^{\vee})^m)^{\op{U}(2,1)}$ is generated by elements of degree 2 for each integer $m\geq 1$ is then exactly Ribet's.
\end{proof}

Now let $\mathcal{A}_t$ be a general fibre of the specific universal abelian scheme described in section \ref{sec:surface}. Then $End^0 (\mathcal{A}_t)=B$ of the type above.

\begin{theorem}
Let $\mathcal{A}_t$ is an abelian 9-manifold with $End^0 (\mathcal{A}_t)=B$ where $B$ be a central simple
division algebra of degree 9 over an imaginary quadratic field $E$ and equipped with a positive involution of the second kind. Suppose that the $E$-action on the tangent space of $\mathcal{A}_t$ is of type $(6,3)$. Then the Hodge conjecture holds for $\mathcal{A}_t$.

\end{theorem}

\begin{proof}
By the Hodge conjecture for degree 2 the Hodge conjecture for any variety $X$ holds if the $(1,1)$-criterion holds for $X$. By theorem \ref{0} we are thus reduced to showing ''$MT(\mathcal{A}_t)=\frak{G}$''. In the proof of theorem \ref{0} we showed that $\frak{G}=\op{GU}(2,1)$. Since we take $\mathcal{A}_t$ to be a general fibre of $\mathcal{A}\ra X$ and by the very construction of the underlying Shimura variety it follows from Deligne's  conditions (2.1.1.1)--(2.1.1.5) of \cite{deligne2} that
$MT(\mathcal{A}_t)=\op{GU}(2,1)$. We can then apply theorem \ref{0} which completes the proof (for $m=1$).
\end{proof}

\section{Cohomology}\label{sec:cohomology}

In this section we prove the following theorem.
\begin{theorem}
Let $\mathcal{V}^i$ denote the following local systems on $X$:
$$\mathcal{V}^i=\left\{
\begin{array}{ll}
R^{i}p_*\C & \textrm{if $i=2r-1, r=0,\ldots ,d$,}\\
R^{i}p_*\C /(R^{i}p_*\C)^{\hod (\mathcal{A}_t)} & \textrm{if $i=2r, r=0,\ldots ,d$,}\\
\end{array}\right.
$$
Then the cohomology groups $H^q(X,\mathcal{V}^i)$ vanish if
$q\neq 2$.
\end{theorem}

\begin{proof}
We start by investigating $R^{i}p_*\C$:
We have

$R^{i}p_*\C =H^{i} (\mathcal{A}_t , \C )$ for a general fibre $\mathcal{A}_t$ with $t\in X$. Base changing $B=End(\mathcal{A}_t)$
to $\C$ yields $M_3(\C)\times M_3(\C)^{opp}$ (see \cite{kottwitz2}). Denote by $V$ the standard representation of $\GL_3$ and by $V^{\vee}$ its dual. For dimensional reasons we thus obtain that $(R^{1}p_*\C)_t =H^{1} (\mathcal{A}_t , \C )$ is the local system associated to $V^{\oplus 3}\oplus V^{\vee\oplus 3}$. For vanishing results it will be sufficient to consider $V\oplus V^{\vee}$. The cohomology of an abelian variety is an outer algebra and we thus obtain
$$H^{i} (\mathcal{A}_t , \C )=\bigwedge^{i}V^{\oplus 3}\oplus V^{\vee\oplus 3}.$$
We will now check the vanishing properties degree-wise for each $(R^{i}p_*\C)_t =H^{i} (\mathcal{A}_t , \C )$. In doing so we will make no notational distinction between a representation and its associated local system. Note that because of duality we only need to investigate $deg= 0,\ldots ,3$.\\
\underline{deg $>$ 6:}\\
Since $\op{dim V} = 3$ the outer product $\bigwedge^{i} V$ will vanish for $i> 3$. Thus $\bigwedge^{i} (V \oplus V^{\vee})$ vanishes for $i > 6$.\\
\underline{deg=0, deg=6:}\\
Clearly $\bigwedge^{0} (V\oplus V^{\vee})=\C$ and this is a Hodge cycle.
\bigskip

We will now collect all the irreducible constituents of $\bigwedge^{i} (V \oplus V^{\vee})$ for each positive degree.

\underline{deg=1, deg=5:}\\
$\bigwedge^{1} (V \oplus V^{\vee})=V \oplus V^{\vee}$ and the standard representation $V$ (and its dual $V^{\vee}$) are irreducible.

\underline{deg=2, deg=4:}\\
$\bigwedge^{2} (V \oplus V^{\vee})=\bigwedge^{2} V \oplus \bigwedge^{2} V^{\vee}\oplus (V \otimes V^{\vee})$. Since
$\bigwedge^{2} V = V^{\vee}$ and $\bigwedge^{2} V^{\vee}=V$
we get new constituents only from $(V \otimes V^{\vee})$.
But $(V \otimes V^{\vee})= \C\oplus \mathfrak{ad}$, where $\C$ is the trivial representation and $\mathfrak{ad}$ denotes the adjoint representation.

\underline{deg=3:}\\
$\bigwedge^{3} (V \oplus V^{\vee})=\bigwedge^{3} V \oplus \bigwedge^{3} V^{\vee}\oplus (\bigwedge^{2} V\otimes V^{\vee})
\oplus (\bigwedge^{2} V^{\vee}\otimes V)$.\\
This yields two copies of the 1-dimensional determinant representation and $V\otimes V$ (respectively $V^{\vee}\otimes V^{\vee}$). $V\otimes V = \bigwedge^{2} V\oplus \op{S}^2 V$
and $\op{S}^2 V$ is irreducible.

\bigskip

Summarizing the local systems we have to deal with, we get $$\{\C_{id}, \C_{det}, V, \mathfrak{ad}, \op{S}^2 V\}$$ (plus dual versions of these).\\
Denote by $(m,r,n)\in\Z^3$ with $m>r>n$ the highest weight of an irreducible representation ${\tau}^0$ of $U(2,1)$ which arises as the restriction of a representation $\tau$ of $GU(2,1)$.\\
$\bullet$ \underline{$\C_{id}$:}\\
As just derived above, $\C_{id}$ occurs once in $H^{2} (\mathcal{A}_t , \C )$ and once in $H^{4} (\mathcal{A}_t , \C )$. 
We claim that these copies of $\C_{id}$ are Hodge cycles.
Indeed recall that  $\C_{id}$ arose as the trivial subrepresentation of $(V \otimes V^{\vee})=End (V)$. We thus may identify it with the identity in $End (V)$ and therefore to the polarization class on the fibres $\mathcal{A}_t$, which is a Hodge class. Since we are only interested in the vanishing of  the cohomology for the local system
$$\mathcal{V}=\left\{
\begin{array}{ll}
R^{i}p_*\C & \textrm{if $i=2r-1, r=0,\ldots ,d$,}\\
R^{i}p_*\C /(R^{i}p_*\C)^{\hod (\mathcal{A}_t)} & \textrm{if $i=2r, r=0,\ldots ,d$,}\\
\end{array}\right.
$$
we can ignore this copy of $\C_{id}$.

$\bullet$ \underline{$\C_{det}$:}\\
As shown in \cite{rogawski}, \S 3.2 a 1-dimensional representation $\tau$ contributes to cohomology only if $m-r=r-n=1$. But the determinant representation has highest weight $(1,1,1)$, hence $m-r=r-n=0$.

$\bullet$ \underline{$\mathfrak{ad}$:}\\
The adjoint representation $\mathfrak{ad}$ (by construction) has regular highest weight and therefore contributes only to the cohomology in middle degree by the results of  Vogan-Zuckerman, see \cite{vog-zuck}.

$\bullet$ \underline{$\op{S}^2 V$:}\\
The highest weight of $V$ is $(1,0,0)$ and the highest weight of $\op{S}^2 V$ is $(2,0,0)$. Since $\op{S}^2 V$ is neither 1-dimensional nor satisfying $m-r=1$ or $r-n=1$ again by \cite{rogawski}, \S 3.2 cohomology can only occur in degree 2.

$\bullet$ \underline{$V$:}\\
Since $dim\neq 1$, cohomology can possibly only occur in degrees 1,2,3. By Clozel's generalization of Rapoport and Zink's vanishing result (loc.cit.) the representation $\pi^n$ of \cite{rogawski}, \S 3.2 doesn't exist
and there is cohomology only in degree 2.

This completes the proof.

\end{proof}

 \begin{rem}
 In the proof the same argument as for {$V$} could have been used for {$\op{S}^2 V$} and {$\mathfrak{ad}$}, but we chose to use this more general reasoning  to show where we actually need
 the very specific arithmetic nature of these Shimura varieties.
 \end{rem}

\section{Proof of the main theorem}\label{sec:proof}

\begin{theorem}
With the above notations $\mathcal{A}$ has a Chow-K\"unneth decomposition over $\C$.
\end{theorem}

\begin{proof}

By the work of Deninger--Murre (see \cite{den-mur}) we know that $\mathcal{A}/X$ has a relative Chow-K\"unneth decomposition which yields assumption (1) of theorem \ref{ghm2}. Since surfaces have Chow-K\"unneth decompositions (see \cite{murre-surface}) we obtain assumption (2). As we have proved the Hodge conjecture for $\mathcal{A}_t$, t generic,  in section \ref{sec:conj} and the Hodge cycles are given precisely by $H_B^{*}(\mathcal{A}_t,\C)^{\hod (\mathcal{A}_t)}$ assumption (4) follows since algebraic cycle classes are even. We furthermore know from the Hodge conjecture for $\mathcal{A}_t$ that there is a surjective map
\[CH^r(\mathcal{A}_t)\to H^{2r}_B(\mathcal{A}_t(\mathbb{C}),\mathbb{Q})^{\hod (\mathcal{A}_t)}\]
 for $0\le r\le d=\dim \mathcal{A}-\dim X$.

 To see that this statement actually implies assumption (3) of our theorem, we note that the restriction map
 $$\chow{r}{\mathcal{A}^{'}}\ra\chow{r}{\mathcal{A}^{'}_{\xi}}$$ is surjective where $\xi$ is a generic point of $\mathcal{A}^{'}.$ Here $\mathcal{A}^{'}$ is a model of $\mathcal{A}$ over $X'$
 where $X'$ is a model of $X$ over the  algebraic closure $k$ of a field of finite type over $\C$. Lifting an element $Z\in\chow{r}{\mathcal{A}^{'}_{\xi}}$ to $\chow{r}{\mathcal{A}^{'}}$ is done by taking the Zariski closure of $Z$. To see that we also get the surjectivity of $\chow{r}{\mathcal{A}}\ra\chow{r}{\mathcal{A}_t}$ with generic $t\in X$, ($X$ viewed as scheme over $\C$), is straight forward, see 1.8. of \cite{gordon-hanamura-murre-iii}.

 Finally concerning assumption (4) the proof of theorem \ref{ghm2} showed that it was sufficient to prove the vanishing assumption for the very specific local systems
$$\mathcal{V}^i=\left\{
\begin{array}{ll}
R^{i}p_*\C & \textrm{if $i=2r-1, r=0,\ldots ,d$,}\\
R^{i}p_*\C /(R^{i}p_*\C)^{\hod (\mathcal{A}_t)} & \textrm{if $i=2r, r=0,\ldots ,d$,}\\
\end{array}\right.
$$
instead of all $\mathcal{V}$. But this is precisely what we have done in section \ref{sec:cohomology}. We thus can apply theorem \ref{ghm2}
 to conclude our claim.

\end{proof}

\section{Acknowledgement}

It is a great pleasure to thank Richard Taylor for many helpful discussions.

\end{document}